\documentclass[11pt,letterpaper]{amsart}

\usepackage{amsmath,amssymb,amsfonts,%amsthm,amsopn,
bbm}
\usepackage{graphicx,
color}
\usepackage{amsmath}
\newcommand\eps{\varepsilon}

\newcommand\logmeno{\mathop{\mathrm{Log}}\nolimits_-}

\newtheorem{theorem}{Theorem}[section]

\newtheorem{corollary}[theorem]{Corollary}
\newtheorem{proposition}[theorem]{Proposition}

\theoremstyle{remark}
\newtheorem{remark}[theorem]{Remark}

\newtheorem*{remark*}{Remark}

\theoremstyle{definition}

\newcommand\bR{\mathbb{R}}
\newcommand\bC{\mathbb{C}}

\newcommand{\rdd}{\mathbb{R}^{2d}}

\newcommand\Hp{\bC_+}  % semipiano superiore
\newcommand\GH{G_H}

\newcommand\cB{\mathcal{B}}

\def\cC{\mathcal{C}}

\newcommand\cV{\mathcal{V}}

\newcommand\kp{\kappa_p}

\numberwithin{equation}{section}

\begin{document}
	
\title[The norm of localization operators]{The norm of time-frequency and wavelet localization operators}
\author{Fabio Nicola}
\address{Dipartimento di Scienze Matematiche, Politecnico di Torino, Corso Duca degli Abruzzi 24, 10129 Torino, Italy.}
\email{fabio.nicola@polito.it}
\author{Paolo Tilli}
\address{Dipartimento di Scienze Matematiche, Politecnico di Torino, Corso Duca degli Abruzzi 24, 10129 Torino, Italy.}
\email{paolo.tilli@polito.it}
\subjclass[2010]{42B10, 49Q20, 49R05, 81S30, 94A12}
%\date{}
\keywords{Short-time Fourier transform, wavelet transform, localization operator, Toeplitz operator, uncertainty principle}
%\date{\today}

\begin{abstract}  Time-frequency localization operators (with Gaussian window) $L_F:L^2(\bR^d)\to L^2(\bR^d)$, where $F$ is a weight in $\bR^{2d}$, were introduced in signal processing by I. Daubechies in 1988, inaugurating a new, geometric, phase-space perspective. Sharp upper bounds for the norm (and the singular values) of such operators turn out to be a challenging issue with deep applications in signal recovery, quantum physics and the study of uncertainty principles.

In this note we provide optimal upper bounds for the operator norm $\|L_F\|_{L^2\to L^2}$, assuming $F\in L^p(\bR^{2d})$, $1<p<\infty$ or $F\in L^p(\bR^{2d})\cap L^\infty(\bR^{2d})$, $1\leq p<\infty$. It turns out that two regimes arise, depending on whether the quantity $\|F\|_{L^p}/\|F\|_{L^\infty}$ is less or greater than a certain critical value. In the first regime the extremal weights $F$, for which equality occurs in the estimates, are certain Gaussians, whereas in the second regime they are proved to be \textit{truncated} Gaussians, degenerating in a multiple of a characteristic function of a ball for $p=1$. This phase transition through truncated Gaussians appears to be a new phenomenon in time-frequency concentration problems. For the analogous problem for wavelet localization operators -where the Cauchy wavelet plays the role of the above Gaussian window- a complete solution is also provided.

\end{abstract}
\maketitle
\section{Introduction}
The uncertainty principle is an ubiquitous theme in mathematics and represents an endless source of challenging and inspirational problems. The literature in this connection is enormous. For a deep introduction to the topic, we address the reader to the classical contributions \cite{cowling, fefferman,folland} and the recent account \cite{wigderson}. 

Roughly speaking, the uncertainty principle states that a function and its Fourier transform cannot both be too concentrated, or equivalently that a time-frequency distribution cannot be too concentrated in the time-frequency space. To better develop the latter point of view, we recall the definition of a time-frequency distribution known in harmonic analysis and signal processing as short-time, or windowed, Fourier transform  \cite{degosson_book,grochenig-book,mallat,muscalu}, and in mathematical physics as coherent state transform \cite{liebloss}.

Let $\varphi$ be the ``Gaussian window''
\begin{equation}
\label{defvarphi}
 \varphi(x)=2^{d/4}e^{-\pi |x|^2}, \quad x\in\bR^d,\end{equation}
normalized in such way that $\|\varphi\|_{L^2}=1$.
The short-time Fourier transform with Gaussian window of a function $f\in L^2(\bR^{d})$, is defined as
\[
\cV f(x,\omega)=
\int_{\bR^d} e^{-2\pi i y\cdot\omega} f(y)\varphi(x-y)dy, \qquad x,\omega\in\bR^d.
\]
It turns out that $\cV:L^2(\bR^d)\to L^2(\bR^{2d})$ is an isometry, so that the function $|\cV f(x,\omega)|^2$ can be interpreted as the time-frequency energy density of $f$. Consequently, the integral
$
\int_\Omega |\cV f(x,\omega)|^2\, dx d\omega,
$
over a measurable subset $\Omega\subset\bR^{2d}$, will be the fraction of its energy trapped in $\Omega$
(see e.g. \cite{abreu2012,abreu2019,abreu2018,galbis2010,romero,seip1991,nicola_tilli}).
Sharp upper bounds for this quantity were recently proved in \cite{nicola_tilli}; the relevant estimate, that we state here in dimension $d=1$ for simplicity, reads
\begin{equation}\label{eq intro 1}
\int_\Omega |\cV f(x,\omega)|^2\, dx d\omega\leq (1-e^{-|\Omega|})\|f\|_{L^2}^2
\end{equation}
(see Theorem \ref{generalFK} for the general statement), where $|\Omega|$ is
the Lebesgue measure of $\Omega\subset\bR^2$. Since when $|\Omega|$ is finite 
the constant on the right-hand side is strictly less than one, this result can be regarded as a manifestation of the uncertainty principle. Upper bounds in the same spirit, for certain Cantor-type rotationally invariant subsets $\Omega\subset\bR^2$ have recently been obtained in \cite{abreu_fractal,knutsen2020,knutsen2022,knutsen_arXiv} in connection with the fractal uncertainty principle. We also address to \cite{bonami,luef_arkiv,galperin,grochenig,nicola_hrt} for other forms
of the uncertainty principle involving the short-time Fourier transform.

While the localization in the time-frequency plane by means of the characteristic function of a subset $\Omega$, as in \eqref{eq intro 1}, is certainly natural, following a practice wich dates back at least to \cite{cowling} and fully promoted in \cite{daubechies}, one can similarly localize using other \textit{weight} functions $F(x,\omega)\geq 0$ (satisfying possibly other constraints such as $F(x,\omega)\leq 1$, to avoid an amplification) and ask for similar estimates for the weighted energy
\[
\langle L_F f,f\rangle =
\int_{\bR^{2d}} F(x,\omega) |\cV f(x,\omega)|^2\, dx d\omega.
\]
Here we introduced the so-called \textit{time-frequency localization operator} $L_F$ associated with the weight $F$, defined as
\begin{equation}\label{eq def LF}
L_F =\cV^\ast F \cV,
\end{equation}
or weakly as 
\begin{equation}\label{eq def LFbis}
\langle L_F f, g\rangle_{L^2(\mathbb{R}^{d})} = \langle F ,\overline{\cV f}\cV g\rangle_{L^2(\mathbb{R}^{2d})}.
\end{equation}
Indeed, if $F\in L^p(\mathbb{R}^{2d})$, $1\leq p\leq\infty$, is any complex-valued function, it is known that $L_F$ is a bounded operator in $L^2(\bR^d)$ (see e.g. \cite{wong}). This class of operators were introduced by I. Daubechies \cite{daubechies} as a joint time-frequency version of the celebrated Landau-Pollack-Slepian operator \cite{landau1985,landau1961,slepian1983} from signal theory, and by F. A. Berezin \cite{berezin}, motivated by the quantization problem in quantum mechanics. Since then, they have been object of intensive studies, especially regarding boundedness, compactness, Schatten properties and asymptotics of the eigenvalues, even for more general weights and function spaces; see the classical references \cite{ berezin-book,daubechies1990,daubechies_book,wong}, the more recent contributions \cite{abreu2016,abreu2017,bayer, cordero,corderonicola2011,fernandez2006,luefJMPA,luef} and the references therein.

In this note we confine ourselves to the basic case of weights in Lebesgue spaces, 
and we prove optimal estimates of the kind
\begin{equation}
\label{stimaC}
\Vert L_F\Vert_{L^2(\bR^2)\to L^2(\bR^d)} \leq C=C(p,A,B)
\end{equation}
for the norm of the localization operator $L_F$, for any function $F\in L^p(\rdd)$
satisfying the double constraint
\begin{equation}
\label{constraints}
\Vert F\Vert_{L^\infty}\leq A\qquad\text{and}\qquad
\Vert F\Vert_{L^p}\leq  B.
\end{equation}
Here and throughout the paper, we assume that $p \in [1,\infty)$, $A\in (0,\infty]$ and $B\in (0,\infty)$, 
with the further condition that $A<\infty$ when $p=1$. Moreover, with some abuse of notation,
the symbol $\Vert F\Vert_{L^\infty}$ is always meant in the broader sense of $\mathop{\rm ess\,sup}|F|$,
i.e. $+\infty$ when $F\not\in L^\infty(\rdd)$. With this proviso, 
the case where $A=\infty$ (and $p>1$) corresponds to dropping the $L^\infty$ constraint:
in this case, optimal bounds can be deduced from Lieb's uncertainty inequality \cite{lieb1978} (see below) and the optimal functions, i.e. those functions $F$ that achieve equality in \eqref{stimaC}, are certain Gaussians. Hence in estimating $\Vert L_F\Vert$, the $L^\infty$
constraint in \eqref{constraints} is only relevant when $A$ is less than the $L^\infty$ norm of these Gaussians. 

% (if $B=+\infty$ the above supremum is trivially equal to $A$, since $L_F$ is the identity operator 
% when $F\equiv 1$). 
We are also able to characterize in all cases the optimal functions $F$, as well as all pairs of functions $f$ and $g$ (normalized in $L^2(\bR^d)$)
such that
\begin{equation}
\label{optfg}
\|L_F\|_{L^2\to L^2}=|\langle L_F f,g\rangle|,
\end{equation}
as Gaussians functions of the kind
\begin{equation}
\label{charf}
x\mapsto c e^{2\pi i x\cdot \omega_0} \varphi(x-x_0),\quad x\in\bR^d,
\end{equation}
for some unimodular $c\in\bC$ and some $(x_0,\omega_0)\in\bR^d\times\bR^d$, 
where $\varphi$ is the Gaussian in \eqref{defvarphi}.

Throughout the paper, we will frequently refer to the constant
\begin{equation}
\label{defkp}\kp:=\frac{p-1}{p},
\end{equation}
without any further reference to \eqref{defkp}.

To give a flavour of our results, in dimension $d=1$ we can state the following theorem
(see Section \ref{sec mult} for the results in arbitrary dimension).
\begin{theorem}\label{thm thm2} Let $p\in [1,\infty)$, $A\in(0,\infty]$, $B\in(0,\infty)$,
with the proviso that $A<\infty$ when $p=1$, and let $F\in L^p(\bR^2)$ satisfy the constraints
\eqref{constraints}.

\begin{itemize}
\item[(i)] If $p=1$, then
\begin{equation}
\label{eq stima0}
\Vert L_F\Vert_{L^2\to L^2}\leq A\big(1-e^{-B/A}\big),
\end{equation}
and equality occurs if and only if, for some $\theta\in\bR$ and some $z_0\in\bR^2$,
\begin{equation}\label{eq estr22}
F(z)=A\,e^{i\theta}\chi_{\cB}(z-z_0)
\qquad \forall z\in\mathbb{R}^2,
\end{equation}
where $\cB\subset\mathbb{R}^2$ is the disc of area $B/A$, centered at the origin.

\item[(ii)] If $p>1$ and $(B/A)^p \leq \kp$ (in particular when $A=\infty$), then
\begin{equation}\label{eq stima1}
\Vert L_F\Vert_{L^2\to L^2}\leq  \kp^{\kp} B,
\end{equation}
with equality if and only if, for some $\theta\in\bR$ and some $z_0\in\bR^2$,
\begin{equation}\label{eq estr1}
F(z)= B \kp^{-\frac 1 p} e^{i\theta}e^{-\frac \pi {p-1}|z-z_0|^2},\qquad z\in\mathbb{R}^2.
\end{equation}

\item[(iii)] If $p>1$ and $(B/A)^p > \kp$, then
\begin{equation}\label{eq stima2}
\Vert L_F\Vert_{L^2\to L^2}\leq A\Big(1-\frac{e^{\kp-(B/A)^p}} p\Big),
\end{equation}
with equality if and only if, for some $\theta\in\bR$ and some $z_0\in\bR^2$,
\begin{equation}\label{eq estr2}
F(z)=e^{i\theta} \min\{\lambda e^{-\frac{\pi}{p-1}|z-z_0|^2},A\},\qquad z\in\bR^2,
%F(z)=c\,\min\{\sigma\gaup(z-z_0),1\}\qquad z\in\mathbb{R}^2
\end{equation}
where $\lambda=A e^{(B/A)^p/(p-1)-1/p}>A$.
\end{itemize}
Finally, if $F$ achieves equality in \eqref{eq stima0}, \eqref{eq stima1}, or
\eqref{eq stima2}, then
\begin{equation*}
%\label{saturF} 
\Vert F\Vert_{L^p} = B,
\end{equation*}
and \eqref{optfg} holds
%\begin{equation}
%\label{optfg}
%\|L_F\|_{L^2\to L^2}=|\langle L_F f,g\rangle|
%\end{equation}
for some $f,g\in L^2(\bR)$ such that $\|f\|_{L^2}=\|g\|_{L^2}=1$, if and only if 
both $f$ and $g$ are Gaussians of the kind \eqref{charf} (with $d=1$), possibly with different 
$c$'s, but with the same $(x_0,\omega_0)\in\bR^2$ given by $z_0$ (the center of symmetry
of $|F|$).
\end{theorem}
As a consequence, given $F\in L^p(\mathbb{R}^2)$, letting $B=\Vert F\Vert_{L^p}$ and $A=\Vert F\Vert_{L^\infty}$, one obtains the following estimates. 
\begin{corollary}\label{thm thm2 bis}\
\begin{itemize}
\item[(a)] Let $1< p<\infty$. For every  $F\in  L^p(\mathbb{R}^2)\setminus\{0\}$ we have
\begin{equation}\label{eq stima1 bis bis}
\|L_F\|_{L^2\to L^2}\leq \kp^{\kp}\,\|F\|_{L^p},
\end{equation}
with equality if and only if $F$ is as described in \eqref{eq estr1}, for some $B>0$, $\theta\in\mathbb{R}$, $z_0\in\mathbb{R}^2$. 
\medskip
\item[(b)] Let $1\leq p<\infty$. For every $F\in L^\infty(\mathbb{R}^2)\cap L^p(\mathbb{R}^2)\setminus\{0\}$, satisfying  $(\|F\|_{L^p}/\|F\|_{L^\infty})^p> \kp$, we have
\begin{equation}\label{eq stima2 bis}
\|L_F\|_{L^2\to L^2}\leq \|F\|_{L^\infty}\Big(1-\frac{e^{\kp-(\|F\|_{L^p}/\|F\|_{L^\infty})^p}}{p}\Big).
\end{equation}

\medskip
If $p>1$, equality occurs in \eqref{eq stima2 bis} if and only if $F$ has the form in \eqref{eq estr2}
for some $\lambda>A>0$, $\theta\in\mathbb{R}$, $z_0\in\mathbb{R}^2$.

\medskip
If $p=1$,  equality occurs in \eqref{eq stima2 bis} if and only if $F=c\,\chi_{\cB}$
for some $c\in\bC\setminus\{0\}$, where $\cB\subset\mathbb{R}^2$ is a ball.
\end{itemize}
%In the above cases of equality, if $F\not\equiv0$, we have $\|L_F\|_{L^2\to L^2}=|\langle L_F f,g\rangle|$ for some $f,g\in L^2(\bR)$, $\|f\|_{L^2}=\|g\|_{L^2}=1$, if and only if for some $c,c'\in\bC$, $|c|=|c'|=1$ and some $(x_0,\omega_0)\in\bR^2$, \eqref{eq deff} holds true.
\end{corollary}
Some remarks are in order.
When $A=\infty$ in Theorem \ref{thm thm2}, the extremal functions $F$ have the form in \eqref{eq estr1} and $\|F\|_{L^\infty}= B \kappa_p^{1/p}$. Hence
 Theorem \ref{thm thm2} (ii) in fact reduces to the case when only the $L^p$ constraints is present. 

%Firstly, when $(\|F\|_{L^p}/\|F\|_{L^\infty})^p=(p-1)/p$,  \eqref{eq stima2} reduces to \eqref{eq stima1}. Instead, when $(\|F\|_{L^p}/\|F\|_{L^\infty})^p> (p-1)/p$, \eqref{eq stima2} represents an improvement upon \eqref{eq stima1} (cf. Remark \ref{rem fut ref}). Also, for the function $F$ in \eqref{eq stima1} we have $(\|F\|_{L^p}/\|F\|_{L^\infty})^p= (p-1)/p$, and this easily implies that no similar improvement upon \eqref {eq stima1} can hold in the regime $(\|F\|_{L^p}/\|F\|_{L^\infty})^p\leq  (p-1)/p$ (see Remark \ref{rem opt} below).

The estimate \eqref{eq stima1 bis bis} is easily seen to be \textit{equivalent}, by duality, to Lieb's uncertainty inequality for the short-time Fourier transform (with Gaussian window) \cite{lieb1978} (see also \cite{boggiatto,carlen,lieb}), which in dimension $d=1$ reads as follows: for $2\leq p<\infty$,
\begin{equation}\label{lieb}
    \|\cV f\|_{L^p(\bR^{2})}\leq (2/p)^{1/p}\|f\|_{L^2(\bR)}.
\end{equation}
Hence, case (b) in Corollary \ref{thm thm2 bis} is the most interesting. However, our proof does not rely on \eqref{lieb} and therefore provides an alternative proof of \eqref{lieb}  as well (we address to \cite{lieb1978} for applications of this estimate in quantum physics).

If $p=1$, \eqref{eq stima2 bis} reduces to the estimate
\begin{equation}\label{eq galbis}
\|L_F\|_{L^2\to L^2}\leq ||F||_{L^\infty}\big(1-e^{-\|F\|_{L^1}/\|F\|_{L^\infty}}\big),
\end{equation}
which was proved in \cite{galbis2021} under the stronger assumptions that $F$ is real valued (so that $L_F$ is self-adjoint) and spherically symmetric, exploiting the fact that in that case $L_F$ diagonalizes in the Hermite bases. Also, when $F$ is the characteristic function of a measurable subset $\Omega$ of finite measure the latter estimate reduces to \eqref{eq intro 1}. Moreover, if $F\in L^1(\mathbb{R}^2)$ and $0\leq F\leq 1$, we deduce that 
$$\|L_F\|_{L^2\to L^2}\leq 1-e^{-{\rm Tr}\, L_F}.$$
Indeed it is well-known that if $F\in L^1$ is non-negative then $L_F$ is trace class and ${\rm Tr}\, L_F=\|F\|_{L^1}$ (see e.g. \cite{wong}). 

In Theorem \ref{thm thm2} we excluded the case $p=1$, $A=\infty$, because in that case the optimal constant in \eqref{stimaC}, which is easily seen to be $C=B$, is not achieved (see Remark \ref{rem p=1}). 

We point out that the appearance of the {\it truncated} Gaussians in \eqref{eq estr2} as extremal 
functions of a time-frequency concentration problem is, to our knowledge, an unprecedented 
fact in the literature. 

The norm estimates in \eqref{eq stima0}, \eqref{eq stima1}, and \eqref{eq stima2} are obtained
through a new intermediate estimate of general interest, valid under no quantitative
restrictions on $F$, that in dimension $d=1$ reads
\begin{equation}
\label{stimanorma1}
\Vert L_F\Vert_{L^2\to L^2}\leq \int_0^\infty \left(1-e^{-\mu(t)}\right)\,dt,
\end{equation}
where $\mu$ is the distribution function of $|F|$ (see Theorem \ref{thm4} for the general
statement, and Corollary \ref{corsym} for a related symmetrization result). Then one may wonder how the right hand side of \eqref{stimanorma1} can
be further estimated if $F$ is subject to \eqref{constraints}, and this leads in a natural
way to a nonstandard variational  problem (where the decreasing function $\mu$ is the
unknown) described and completely solved in Section \ref{sec due}. Since
the solution to this variational problem is (depending on the values of $p$, $A$ and $B$) the distribution 
function of a ball, of a Gaussian, or of a truncated Gaussian, this eventually
leads to Theorem \ref{thm thm2} (and Theorem \ref{thm thm2 mult} in arbitrary dimension).

%Also, the presence of two regimes according to a critical value for  $\|F\|_{L^p}/\|F\|_{L^\infty}$ arises naturally in the variational perspective of Theorem \ref{thm thm2}. Indeed, by a slicing argument and using the concentration estimates \eqref{eq intro 1} on the superlevel sets of $|F|$, we are led to an optimization problem for the distribution function $u(t)$ of $|F|$, which will be solved first, in Section \ref{sec due}. The condition $(B/A)^p>  \kp$ in Theorem \ref{thm thm2} (iii) (in fact, the condition $(B/A)^p\ge   \kp$) then corresponds to the case when the optimal $u$ is strictly positive in $(0,A)$. 

We also notice that the above estimates could be rephrased as uncertainty principles. As an illustration, suppose that for some $\varepsilon\in (0,1)$, $f\in L^2(\bR)\setminus\{0\}$ and $F\in L^\infty(\bR^2)$, with $0\leq F\leq 1$, we have
\[
\int_{\bR^2}F(x,\omega)|\cV f(x,\omega)|^2\, dx d\omega\geq (1-\epsilon)\|f\|^2_{L^2}.
\]
Then it follows from \eqref{eq galbis} that $\|F\|_{L^1}\geq \log(1/\varepsilon)$, which generalizes \cite[Theorem 1.2]{nicola_tilli}, where the case $F=\chi_\Omega$ was considered.

The above results could also be interpreted as optimal bounds for the norm of Toeplitz operators in the Bargmann-Fock space of entire functions \cite{zhu} (cf. \cite{galbis2021}).
Also, this general scheme seems to go to the very heart of the problem and is flexible enough to encompass more general function spaces, where the norm $\|F\|_{L^p}^p$ is replaced by $\int_{\bR^{2d}} \Phi(|F(x,\omega)|)\,dxd\omega$, $\Phi$ being a non-negative convex function satisfying some natural conditions, and the objective function is also of the type $\int_{\bR^{2d}} F(x,\omega)G(|\cV f(x,\omega)|)\, dx d\omega$, providing a straightforward pattern to find optimal bounds and corresponding optimal functions. Here we chose to confine ourselves to the case of the $L^p$ setting for the sake of concreteness. However, to illustrate the scope of this approach, in Section \ref{sec wavelet} we briefly prove similar results for wavelet localization operators \cite{daubechies_book,daubechies_paul}. In that case, the Cauchy wavelet plays the role of the Gaussian window \eqref{defvarphi}, and the estimate analogous to \eqref{eq intro 1} was recently proved in \cite{joao_tilli}. 

We emphisize that the concentration problem for the wavelet transform is directly related to deep issues in the theory of Bergman spaces; for example, the bound analogous of Lieb's uncertainty inequality --which we also recapture in Theorem \ref{thm thm2 wav} (ii) below (case $A=\infty$)-- has recently appeared in \cite{kulikov} in the form of an optimal contractivity estimate for such spaces. We address the interested reader to \cite{abreu22,seip2018} for a general account of this circle of ideas, and to \cite{lieb2021} and the references therein for the intimate connection with the Wehrl conjecture by Lieb and Solovej. 

\section{The multidimensional case}\label{sec mult}
The extension of Theorem \ref{thm thm2} to arbitrary dimension $d\geq1$
requires
the introduction of the function
\begin{equation}
\label{defG}
G(s):=
%\frac{1}{(d-1)!}\int_0^{\pi(s/\bo)^{1/d}}  \tau^{d-1}e^{-\tau}\, d\tau,
\int_0^s e^{-(d! \tau)^{\frac 1 d}}\,d\tau
\end{equation}
%where $\bo$ is the volume of the unit ball in $\bR^{2d}$ 
(note that $G(s)=1-e^{-s}$ when $d=1$).
With this notation, we can now state the following
\begin{theorem}\label{thm thm2 mult}
Assume $p\in [1,\infty)$, $A\in (0,\infty]$ and $B\in (0,\infty)$, with the proviso
that $A<\infty$ if $p=1$, and let $F\in L^p(\bR^{2d})$ satisfy the constraints \eqref{constraints}.
%be such that 
%\begin{equation}
%\label{constraints}
%\Vert F\Vert_{L^\infty}\leq A\qquad\text{and}\qquad
%\Vert F\Vert_{L^p}\leq  B.
%\end{equation}
\begin{itemize}
\item[(i)] If $p=1$, then
\begin{equation}\label{eq stima2 ter 2}
 \|L_F\|_{L^2\to L^2}\leq  A \,\,G(B/A)
 \end{equation}
where $G(s)$ is as in \eqref{defG}, and equality occurs if and only if, for some $\theta\in\bR$
and some $z_0\in\rdd$,
\begin{equation}
\label{Fcasei}
F(z)=A e^{i\theta} \chi_{\cB}(z-z_0)\quad\forall z\in\rdd,
\end{equation}
where $\cB\subset\rdd$ is the ball of volume $B/A$, centered at the origin.
\item[(ii)] If $p>1$ and $(B/A)^p\leq \kp^d$, then
\begin{equation}\label{eq stima1 mult 2}
\|L_F\|_{L^2\to L^2}\leq \kp^{d\kp} B,
\end{equation}
with equality if and only if, for some $\theta\in\bR$ and some
$z_0\in\rdd$,
\begin{equation}\label{eq estr1 mult}
F(z)= e^{i \theta} \lambda e^{-\frac{\pi}{p-1}|z-z_0|^2}\qquad z\in\rdd,
\end{equation}
where $\lambda= \kp^{-d/p}B$.
\item[(iii)] If $p>1$ and $(B/A)^p > \kp^d$, then 
\begin{equation}\label{eq stima2 mult}
 \|L_F\|_{L^2\to L^2}\leq \int_0^{A}G(u_{\lambda}(t))\, dt,
\end{equation}
where $u_\lambda(t)=(-\log ((t/\lambda)^{p-1}))^d/d!$ and $\lambda>A$
is uniquely determined by the condition that $p\int_0^A t^{p-1}u_\lambda(t)\,dt=B^p$.
Equality occurs in \eqref{eq stima2 mult} if and only if, for some $\theta\in\bR$ and some
$z_0\in\rdd$,
\begin{equation}\label{eq estr2 mult 2}
F(z)=e^{i\theta} \min\{\lambda e^{-\frac{\pi}{p-1}|z-z_0|^2},A\}\qquad z\in\rdd.
\end{equation}
\end{itemize}

Finally, if $F$ achieves equality in \eqref{eq stima2 ter 2}, \eqref{eq stima1 mult 2}, or
\eqref{eq stima2 mult}, then
\begin{equation}
\label{saturF} \Vert F\Vert_{L^p} = B,
\end{equation}
and \eqref{optfg} holds
%\begin{equation}
%\label{optfg}
%\|L_F\|_{L^2\to L^2}=|\langle L_F f,g\rangle|
%\end{equation}
for some $f,g\in L^2(\bR^d)$ such that $\|f\|_{L^2}=\|g\|_{L^2}=1$, if and only if 
both $f$ and $g$ are functions of the kind \eqref{charf}, possibly with different 
$c$'s, but with the same $(x_0,\omega_0)\in\rdd$ that coincides with $z_0$ (the center of symmetry
of $|F|$).
\end{theorem}
The proof of Theorem \ref{thm thm2 mult}, which will be given Section \ref{secdim},
partially relies on the following result, which is of standalone interest
and provides an explicit bound for the norm of
the operator $L_F$, in terms of the distribution function of $|F|$.

\begin{theorem}
\label{thm4}
Assume $F\in L^p(\bR^{2d})$ for some $p\in [1,+\infty)$, and let 
\begin{equation}
\label{defmu}
\mu(t)=\left\vert\left\{ |F|>t\right\}\right\vert,\quad t>0,
\end{equation}
denote the distribution function of $|F|$. Then
\begin{equation}
\label{stimanorma}
\Vert L_F\Vert_{L^2\to L^2}\leq \int_0^\infty G\bigl(\mu(t)\bigr)\,dt
\end{equation}
where $G$ is as in \eqref{defG},  in particular when $d=1$
\begin{equation*}
%\label{stimanorma1}
\Vert L_F\Vert_{L^2\to L^2}\leq \int_0^\infty \left(1-e^{-\mu(t)}\right)\,dt.
\end{equation*}
Equality occurs in \eqref{stimanorma} if and only if $F(z)=e^{i\theta} \rho(|z-z_0|)$ 
for some $\theta\in\bR$, some $z_0\in\rdd$, and some nonincreasing function $\rho:[0,+\infty)\to[0,+\infty)$.
% someis (up to multiplication by a unimodular
%constant) nonnegative, radially symmetric around some $z_0\in \bR^{2d}$, and radially decreasing.
In this case,  \eqref{optfg} holds true
%\begin{equation}
%\label{fandg}
%\left\vert\langle L_F f,g\rangle\right\vert =\Vert L_F\Vert_{L^2\to L^2}
%\end{equation}
for some $f,g\in L^2(\bR^d)$ such that $\Vert f\Vert_{L^2}=\Vert g\Vert_{L^2}=1$, if (and only
if, when  $F$ is not identically zero) both $f$ and $g$ are functions of the kind \eqref{charf}, possibly with different 
$c$'s, but with the same $(x_0,\omega_0)$ that coincides with $z_0$.
\end{theorem}

To prove Theorem \ref{thm4}, we need the following result (\cite[Theorem 1.1]{nicola_tilli}).
\begin{theorem}[\cite{nicola_tilli}]\label{generalFK}
For every $f\in L^2(\bR^d)$ such that $\|f\|_{L^2}=1$, and every measurable subset $\Omega\subset\mathbb{R}^{2d}$ of measure $|\Omega|<\infty$, we have
\begin{equation}
\label{FKgen}
\int_{\Omega} |\cV f(x,\omega)|^2\, dxd\omega\leq G\bigl(|\Omega|\bigr)
\end{equation}
where $G$ is as in \eqref{defG},  in particular when $d=1$
\begin{equation}
\label{FKdim1}
\int_{\Omega} |\cV f(x,\omega)|^2\, dxd\omega\leq 1-e^{-|\Omega|}.
\end{equation}
Equality occurs in \eqref{FKgen} (for some set $\Omega$ such that $0<|\Omega|<\infty$ and some
$f$ such that $\|f\|_{L^2}=1$)
if and only if $f$ is a function of the form as in \eqref{charf}
%\begin{equation}
%\label{charf}
%x\mapsto c e^{2\pi i x\cdot \omega_0} \varphi(x-x_0),\quad x\in\bR^d,
%\end{equation}
%for some $c\in\bC$ with $|c|=1$ and some $(x_0,\omega_0)\in\bR^d\times\bR^d$, 
and $\Omega$ is equivalent, in measure, to a ball of center $(x_0,\omega_0)$.
%, where $\varphi$ is the Gaussian in \eqref{defvarphi}.
\end{theorem}
We warn the reader that in \cite{nicola_tilli} \eqref{FKgen} is stated in a different
way, namely with $G(s)$ defined as
\[
G(s)=\frac 1 {(d-1)!} \int_0^{\pi(s/\omega_{2d})^{1/d}} t^{d-1}e^{-t}\,dt,
\]
where $\omega_{2d}$ is the volume of the unit ball in $\rdd$. Since  $\omega_{2d}=\pi^d/d!$,
by the change of variable $t^d=d!\,\tau$ one can easily check that this definition coincides with
that in \eqref{defG}.

\begin{proof}[Proof of Theorem \ref{thm4}] Given $f$ and $g$ normalized in $L^2(\bR^d)$, writing $z$ to denote the
variable $(x,\omega)\in\rdd$ and abbreviating $dz$ for $dxd\omega$, we have, by \eqref{eq def LFbis}, 
\begin{equation}
\label{eqeq31}
\begin{aligned}
&%\|L_F\|_{L^2\to L^2}=
\left\vert \langle L_F f,g\rangle\right\vert \leq \int_{\mathbb{R}^{2d}}|F(z)|\cdot
|\cV f(z)|\cdot |{\cV g(z)}|\, dz\\
&\leq \Big(\int_{\mathbb{R}^{2d}}|F(z)|\cdot |\cV f(z)|^2\, dz\Big)^{1/2}\Big(\int_{\mathbb{R}^{2d}}|F(z)|\cdot|\cV g(z)|^2\, dz\Big)^{1/2}.
\end{aligned}
\end{equation}
Letting $m=\mathop{\rm ess\, sup}|F(z)|$ (and assuming $m>0$, otherwise the claims to be proved
are trivial),
we apply the ``layer cake" formula (cf. \cite[Page 26]{liebloss})
\begin{equation}
\label{lcake}
|F(z)|=\int_0^{m}\chi_{\{|F|>t\}}(z)\, dt,\qquad z\in\rdd,
\end{equation}
and use \eqref{FKgen} (with $\Omega=\{|F|>t\}$) to estimate
\begin{equation}
\label{stima1}
\begin{aligned}
\int_{\mathbb{R}^{2d}}|F(z)|\cdot|\cV f(z)|^2\, dz&= \int_0^{m}
\left(\int_{\{|F|>t\}}|\cV f(z)|^2\, dz\right)\,dt\\
& 	\leq \int_0^{m} G(\mu(t))\, dt
\end{aligned}
\end{equation}
(since $G(0)=0$, the last integral coincides with the one in \eqref{stimanorma}, by the definition of $m$).
Since the same argument, with $g$ in place of $f$, leads to
\begin{equation}
\label{stima2}
\int_{\mathbb{R}^{2d}}|F(z)|\cdot|\cV g(z)|^2\, dz
 	\leq \int_0^{m} G(\mu(t))\, dt,
\end{equation}
from \eqref{eqeq31} and the arbitrariness of $f$ and $g$
we obtain \eqref{stimanorma}. 
In fact, since $p<\infty$,   $L_F$ is a compact operator on $L^2(\bR^d)$ (see \cite{wong}), 
hence there exist $f,g$ (normalized in $L^2(\mathbb{R}^d)$) satisfying \eqref{optfg},
so that having equality in \eqref{stimanorma} is equivalent to having equality (with these $f$ and $g$)
in \eqref{eqeq31}, \eqref{stima1}
and \eqref{stima2}:
we now characterize for what $F$, $f$ and $g$
this really occurs.

%First, by Theorem \ref{thm thm1}, equality in \eqref{stima23} occurs
%if and only if $\mu(t)=u(t)$ for every $t\in (0,A)$, where $u$ is one of the functions defined
%in \eqref{ucasoa}, in \eqref{ucasoa2} or in \eqref{ucasob} (depending on $p$, $A$ and $B$).
%Since $\mu(t)=0$ for $t\geq A$, this completely determines $\mu$, which in particular
%is continuous in $(0,A)$. 

First,
equality occurs in \eqref{stima1} if
and only if
\begin{equation}
\label{eqt}
\int_{\{|F|>t\}}|\cV f(z)|^2\, dz= G(\mu(t))
\end{equation}
for a.e. $t\in (0,m)$. When this happens, the validity of \eqref{eqt} for just \emph{one} $t_0\in (0,m)$
is enough, by Theorem \ref{generalFK}, to infer
that $f$ has the form as in \eqref{charf}
(for some $(x_0,\omega_0)\in\mathbb{R}^{2d}$ and
some unimodular $c$), and that the corresponding level set 
$\{|F|>t_0\}$ is (equivalent to) a ball centered at $z_0:=(x_0,\omega_0)$; then,
still by Theorem \ref{generalFK}, the fact that \eqref{eqt} holds  for a.e. $t\in (0,m)$
implies ($f$ being the same) that also all the other levels sets $\Omega_t=\{|F|>t\}$ are equivalent to
balls centered at the same $z_0$ (once this is known for a.e. $t\in (0,m)$, the passage
to \emph{every} $t\in (0,m)$ follows since $\Omega_t=\bigcup_{s>t}\Omega_s$).
These necessary conditions on $f$ and $F$ are, in turn, also sufficient to guarantee \eqref{eqt}
for a.e. $t\in (0,m)$. Thus, summing up, equality in \eqref{stima1} is equivalent
to $f$ being as in \eqref{charf} and $|F(z)|$ being sphericall symmetric (and radially
decreasing) around $z_0$, i.e. $|F(z)|=\rho(|z-z_0|)$ for some nonincreasing function
$\rho:[0,+\infty)\to[0,+\infty)$.

Equality in \eqref{stima2} can be characterized in a similar fashion, so that
also $g$ must be as in \eqref{charf}
(possibly with a different $c$, but with same $z_0$ 
because $F$ is the same).

Therefore, as soon as equality occurs in \eqref{stima1} and \eqref{stima2},
we have $\cV g(z)=e^{i\alpha}\cV f(z)$ for some $\alpha\in\bR$, so that the second inequality in \eqref{eqeq31}
is automatically an equality, while
the first inequality, which explicitly amounts to
\[
\left|  \int_{\mathbb{R}^{2d}}F(z)\,
\cV f(z)\, \overline{\cV g(z)}\, dz
\right|
\leq
\int_{\mathbb{R}^{2d}}|F(z)|\cdot
|\cV f(z)|\cdot |{\cV g(z)}|\, dz,
\]
becomes an equality if and only if
\[
e^{-i\theta} \int_{\mathbb{R}^2}F(z) \cdot |\cV f(z)|^2\, dz= 
\int_{\mathbb{R}^2}|F(z)|\cdot |\cV f(z)|^2\, dz
\]
for some $\theta\in\bR$. This, in turn, is equivalent to the condition
\[
e^{-i\theta}\, F(z)\cdot |\cV f(z)|^2= |F(z)|\cdot |\cV f(z)|^2\qquad \text{for a.e. $z\in\rdd$,}
\]
but since $|\cV f(z)|^2>0$ (note that $|\cV f(z)|^2$ is a Gaussian when $f$ is as in \eqref{charf}) 
we see that equalities occur in \eqref{eqeq31}, for $f$ and $g$ as in \eqref{charf}, if and only if
$F(z)=e^{i\theta} |F(z)|$ (i.e. $F(z)=e^{i\theta}\rho(|z-z_0|)$) for a.e. $z\in\rdd$. 
\end{proof}
An interesting consequence is the following symmetrization result, which shows that
the norm of the localization operator $L_F$ increases, when $F$ is replaced by its
Schwarz symmetrization (we refer to \cite{librosymm} for a general account on
symmetrization).
\begin{corollary}\label{corsym}
Assume $F\in L^p(\bR^{2d})$ for some $p\in [1,+\infty)$, and let 
$F^*$ denote the Schwarz symmetrization of $|F|$. Then
\begin{equation}
\label{stimanormasym}
\Vert L_F\Vert_{L^2\to L^2}\leq
\Vert L_{F^*}\Vert_{L^2\to L^2},
\end{equation}
and equality occurs if and only if $F(z)=e^{i\theta} F^*(z-z_0)$ for some $\theta\in\bR$
and some $z_0\in\bR^{2d}$.
\end{corollary}
\begin{proof}
Let $\mu(t)$ be defined as in \eqref{defmu}, and observe that $\mu$ is also the distribution
function of $F^*$. Applying Theorem \ref{thm4} to $F^*$, we get the bound
\eqref{stimanorma} with $F^*$ in place of $F$. But since $F^*(z)$ is, by construction,
of the form $\rho(|z|)$ where $\rho:[0,+\infty)\to[0,+\infty)$ is nonincreasing,
by Theorem \ref{thm4} this bound is in fact an equality, i.e.
\begin{equation}
\label{boundF*}
\Vert L_{F^*}\Vert_{L^2\to L^2}= \int_0^\infty G\bigl(\mu(t)\bigr)\,dt,
\end{equation}
which combined with \eqref{stimanorma} yields \eqref{stimanormasym}. Equality, by
Theorem \ref{thm4}, occurs therein if and only if $F(z)=e^{i\theta} \tilde\rho(|z-z_0|)$, for 
some $\theta\in\bR$, some
$z_0\in\bR^{2d}$, and some nonincresing function $\tilde\rho:[0,+\infty)\to[0,+\infty)$.
Finally, since $\tilde\rho(|z-z_0|)=|F(z)|$  has the same distribution function as $\rho(|z|)=F^*(z)$, 
we see that $\tilde\rho=\rho$, hence equality in \eqref{stimanormasym} is equivalent
to $F(z)=e^{i\theta}F^*(z-z_0)$
for a.e. $z\in\bR^{2d}$.
\end{proof}

\section{A nonstandard variational problem}\label{sec due}

To prove Theorems \ref{thm thm2 mult} and \ref{thm thm2}
we shall build on Theorem \ref{thm4}, seeking sharp upper bounds for the right hand side of \eqref{stimanorma}
when $F$ is subject to the double constraint in \eqref{constraints}. Given $p$, $A$ and $B$
as in Theorem \ref{thm thm2 mult}, based on the integral constraint
\begin{equation}\label{eq due}
p\int_0^{A} t^{p-1} u(t) \, dt\leq B^p
\end{equation}
we define the class of functions
\begin{equation}
\label{defC}
\cC=\left\{ u:(0,A)\to[0,+\infty)\,|\,\,\text{$u$ is descreasing and satisfies \eqref{eq due}}\right\}.
\end{equation}
If $\mu$ is the distribution function of $|F|$ as in \eqref{defmu}, where
$F\in L^p(\rdd)$ and satisfies \eqref{constraints},
the $L^\infty$ constraint
can be expressed as
\begin{equation}
\label{boundmu}
\mu(t)=0\quad\forall t\geq A,
\end{equation}
while the $L^p$ constraint corresponds to letting $u=\mu$ in \eqref{eq due}.
Hence, since $\mu$ is decreasing, we see that $\mu$ (restricted to the relevant interval $(0,A)$)
belongs to the class $\cC$ as defined above,
and to estimate the right-hand side of \eqref{stimanorma} 
it is natural to investigate the variational problem
\begin{equation}\label{eq uno}%anche \label{eq tre}
\sup_{v\in\cC}I(v)\quad\text{where}\quad
    I(v):= \int_0^{A} G(v(t))\, dt
\end{equation}
(integration in
\eqref{eq uno} can be restricted to $(0,A)$ due to \eqref{boundmu}, since $G(0)=0$).
We first show the existence of an extremal function.
\begin{proposition}\label{pro pro1}
The supremum in \eqref{eq uno} is finite and is attained by at least one function $u\in\cC$.
Moreover, every extremal function $u$ achieves equality in the constraint  \eqref{eq due}.
\end{proposition}
\begin{proof}
Due to \eqref{eq due}, for every $u\in\cC$ and every $t\in (0,A)$ we have
\[
t^p u(t)\leq p\int_0^{t} \tau^{p-1} u(\tau)\, d\tau\leq B^p,
\]
which yields the pointwise bound
\begin{equation}
\label{bound1}
u(t)\leq B^p/t^p\qquad\forall t\in (0,A),\quad
\forall u\in \cC.
\end{equation}
Since $G$ in \eqref{defG} is increasing,  from this bound we obtain
\begin{equation}\label{dom}
I(u)=\int_0^A G(u(t))\,dt\leq
\int_0^A G\bigl(B^p/t^p\bigr)\,dt<\infty \quad\forall u\in\cC
\end{equation}
(when $A=\infty$, and hence $p>1$, the finiteness of the last integral follows from the bound $G(s)\leq s$).
As a consequence, the supremum in \eqref{eq uno} is finite.

Now let $u_n\in\cC$ be a maximizing sequence for problem \eqref{eq uno}, i.e.
\[
\lim_{n\to\infty} I(u_n)=\sup_{v\in\cC}I(v).
\]
Since each $u_n$ satisfies a bound as in \eqref{bound1},
by Helly's selection
theorem we can extract a subsequence (still denoted by $u_n$) pointwise converging to
a decreasing function $u: (0,A)\to [0,+\infty)$, which still satisfies \eqref{eq due}
by Fatou's lemma, 
so that $u\in\cC$ as well. As already observed, $u_n(t)\leq B^p/t^p$, and therefore by
dominated convergence (arguing as in \eqref{dom}) we have
$
I(u)=\lim_{n\to\infty} I(u_n)$,
which proves that $u$ is a maximizer.

Now assume that a strict inequality occurs in the constraint \eqref{eq due}.
Then, letting $u_\varepsilon(t)=u(t)+\varepsilon e^{-t}$ and choosing $\varepsilon>0$ small enough, we
would have $u_\varepsilon\in\cC$: since the function $G(s)$ in \eqref{defG} is strictly increasing,
we would also have $I(u_\varepsilon)>I(u)$, which is impossible since $u$ is a maximizer.
\end{proof}
We now show that one can remove the monotonicity assumption in \eqref{defC},
while retaining the same extremal functions for the corresponding maximization problem. 

%Indeed, let $\cC'$ denote the set of all measurable functions $u:(0,A)\to[0,+\infty)$ satisfying the constraint \eqref{eq due}, and consider the maximization problem
%\begin{equation}\label{eq quattro}
%\sup_{v\in \cC'} I(v).
%\end{equation}

\begin{proposition}\label{pro pro2}
We have
\begin{equation}
\label{supuguali}
\sup_{v\in \cC} I(v)=\sup_{v\in \cC'} I(v),
\end{equation}
where
\[
\cC'=\left\{ u:(0,A)\to[0,+\infty)\,|\,\,\text{$u$ is measurable and satisfies \eqref{eq due}}\right\}.
\]
In particular, any function $u\in\cC$ achieving the supremum on the left-hand side,
also achieves the supremum on the right-hand side.
\end{proposition}
\begin{proof} Consider an arbitrary $u\in\cC'$, and let
$u^\ast$ denote  the decreasing rearrangement of $u$ (i.e. 
the unique descreasing, right-continuous function on $(0,A)$ which is
equimeasurable with $u$, see \cite[Section 10.12]{HLP}).
 Since for every $s>0$
the superlevel set $\{u>s\}\subseteq(0,A)$  has finite measure (this is trivial
when $A<\infty$, and is guaranteed by \eqref{eq due} when $A=\infty$),  we 
see that $u^\ast:(0,A)\to [0,+\infty)$ takes only finite values. 
We also claim that
\begin{equation}
\label{arrang}
p\int_0^{A} t^{p-1} u^\ast (t)\, dt\leq p\int_0^{A} t^{p-1} u (t)\, dt.
\end{equation}
Indeed, let $\nu$ denote the Radon measure on $(0,A)$ with density $ t^{p-1}$,
and observe that, since $t^{p-1}$ is an increasing function of $t$,
\[
\nu(E)=\int_E  t^{p-1}\,dt  \geq
\int_0^{|E|} t^{p-1}\,dt=\nu\bigl((0,|E|)\bigr)
\]
for every measurable $E\subseteq (0,A)$, where $|E|$ denotes the Lebesgue measure of $E$.
In particular, when $E=\{u>s\}$  is a superlevel set of $u$, we obtain
\[
\nu(\{u>s\})\geq \nu\bigl((0,|\{u>s\}|)\bigr)
=\nu\bigl((0,|\{u^*>s\}|)\bigr)
=\nu(\{u^\ast>s\})\quad\forall s\geq 0
\]
(the first equality is due to the equimeasurability of $u$ and $u^*$, the second to the fact
that $u^*$ is decreasing and right-continuous).
Then \eqref{arrang} follows, since
\begin{align*}
 &\int_0^A t^{p-1} u(t)\,dt=
\int_0^A u(t)\,d\nu(t)=
\int_0^\infty \nu(\{u>s\})\,ds\\
&\geq\int_0^\infty \nu(\{u^*>s\})\,ds
=\int_0^A u^*(t)\,d\nu(t)=
 \int_0^A t^{p-1} u^*(t)\,dt.
\end{align*}
Since $u^\ast$ is decreasing and $u\in\cC'$, \eqref{arrang} shows that $u^\ast\in\cC$.
On the other hand, $u$ and $u^\ast$ are equi-measurable, so that
$I(u)=I(u^\ast)\leq \sup_{v\in\cC} I(v)$.
Then, from the arbitrariness of $u\in\cC'$, we see that the inequality $\geq$ occurs in
\eqref{supuguali}, while the opposite inequality is trivial because $\cC\subset \cC'$.
\end{proof}
\begin{remark} This proof entails that if $u\in\cC'$ is not
equal a.e. to a decreasing function (i.e. if $u$
and $u^*$ do not coincide a.e.), then the inequality in \eqref{arrang}
is strict, and hence (by the last part of Proposition \ref{pro pro1}) $u^*$ cannot achieve
the supremum in \eqref{eq uno}. Since $I(u)=I(u^*)$, 
from \eqref{supuguali} we infer that the second supremum in \eqref{supuguali} can be
achieved only by functions in $\cC$.
In the following, however, we will not need this fact.
\end{remark}
Now we are in a position to completely solve problem \eqref{eq uno}. Let
\begin{equation}\label{def log-}
\logmeno(x):=\max\{-\log x,0\},\quad x>0. 
\end{equation}
\begin{theorem}\label{thmcharmax}
There exists a unique function $u\in\cC$ attaining the  supremum in \eqref{eq uno}, and
this
$u$ achieves equality in \eqref{eq due}. More precisely:
\begin{itemize}
\item[(i)] If $p=1$ (hence $A<\infty$), then $u(t)=B/A$ is constant on $(0,A)$, and
\begin{equation}
\label{Icasoi}
I(u)=%\sup_{v\in\cC} I(v)=
A \, G(B/A)
\end{equation}
where $G(s)$ is defined as in \eqref{defG}.
%
%
%while if $p>1$
%it has the form
% \begin{equation}\label{eq mult2*}
%u(t)=\frac 1 {d!}\max\Big\{(-\log(\lambda t^{p-1}))^d,0\Big\}\qquad\forall t\in (0,A),
% \end{equation}
% for some constant $\lambda>0$ uniquely determined by the fact that $u$ satisfies the constraint \eqref{eq due} with equality.
% , i.e.
% \begin{equation}\label{eq mult3}
%     p\int_0^A t^{p-1} u_\lambda(t)\, dt=B^p.
% \end{equation}
%
%\begin{equation}
%\label{ulog}
%%u(t)=\max\{-\log(t/\lambda)^{p-1},0\}\qquad t\in (0,A),
%\end{equation}
%where $\lambda>0$ is a constant that depends on $p$, $A$ and $B$. More precisely:
\item[(ii)] If $p>1$ and $(B/A)^p \leq \kp^d$ (in particular when $A=\infty$), then 
\begin{align}
\label{ucasob}
u(t)=
\frac 1 {d!}\bigl(\logmeno((t/\lambda)^{p-1})\bigr)^d,\quad\forall t\in (0,A)
\end{align}
where $\lambda=B \kp^{-\frac d p}\leq A$, and
\begin{align}
I(u)=%\sup_{v\in\cC} I(v)= 
\kp^{d \kp}\,B.
\label{Icasob}
\end{align}
\item[(iii)] If $p>1$ and $(B/A)^p > \kp^d$, then $u$ is given by
\begin{equation}
\label{ucasoa}
u(t)=\frac 1 {d!}\bigl(-\log ((t/\lambda)^{p-1})\bigr)^d\quad\forall t\in (0,A),
\end{equation}
where $\lambda> A$ is uniquely determined by the condition that $u$ achieves equality in \eqref{eq due}.
%\begin{equation}
%\label{ucasoad=1}
%u(t)=(B/ A)^p-\kp -\log(t/A)^{p-1}\quad
%\forall t\in (0,A),
%%\quad \sigma=e^{(B/A)^{p}/(p-1)-1/p}.
%\end{equation}
Moreover, when $d=1$, $\lambda=A e^{(B/A)^p/(p-1)-1/p}$ and
\begin{equation}
\label{Icasoad=1}
I(u)=%\sup_{v\in\cC} I(v)=
A\left(1-\frac{e^{\kp-(B/A)^p}}p\right)\qquad\text{(when $d=1$).}
\end{equation}
%\end{equation}

\end{itemize}
\end{theorem}
\begin{proof}[Proof of Theorem \ref{thmcharmax}]
We already know from Proposition \ref{pro pro1} that the supremum in \eqref{eq uno} is
achieved by some $u\in\cC$ and that equality holds in \eqref{eq due} for any such $u$, i.e.
\begin{equation}
\label{saturated*}
p \int_0^A t^{p-1} u(t)\,dt=B^p.
\end{equation}
 Moreover, by Proposition \ref{pro pro2}, any such $u$ also achieves
 the second supremum in \eqref{supuguali}: we now exploit this information
to determine $u$.

Let us first assume that $p=1$ (case (i)).
 By Jensen
inequality and \eqref{saturated*} 
\[
\frac {I(u)}A =\frac 1 A \int_0^A G\left(u(t)\right)\,dt
\leq G\left(
\frac 1 A \int_0^A u(t)\,dt\right)
=G(B/A)
\]
with equality if and only if $u$ is constant (i.e. $u= B/A$), since
$G$ is strictly concave. Then \eqref{Icasoi} follows, and case (i) is proved.

\smallskip

From now on, we assume that $p>1$.
Observe that $u$ is not identically zero, due to \eqref{saturated*}.
Thus, letting
\begin{equation}
\label{defM*}
M=\sup \,\{ t\in (0,A)\,|\,\, u(t)>0\},
\end{equation}
since $u$ is decreasing  and $u\geq 0$, we have $M\in (0,A]$
and $u(t)>0$ for $t\in (0,M)$, while $u(t)=0$ for $t\in (M,A)$ if $M<A$. 
%In any case,
%\eqref{saturated*} reduces to
%\begin{equation}
%\label{saturated3}
%p \int_0^M t^{p-1} u(t)\,dt=B^p.
%\end{equation}

Now fix two numbers $a,b$ such that $0<a<b<M$, and let $\eta\in L^\infty(0,A)$ be an arbitrary
function supported in $[a,b]$ and such that
\begin{equation}
\label{ortog*}
\int_a^b t^{p-1} \eta(t)\, dt=0.
\end{equation}
Since on $[a,b]$ we have $u(t)\geq u(b)>0$, if $|\varepsilon|$ is small enough
(depending on $\Vert\eta\Vert_{L^\infty}$) we have   $u+\varepsilon \eta\in \cC'$,
and the function $I(u+\varepsilon\eta)$ has a maximum
at $\varepsilon=0$ since $u$ is a maximizer of $I$ on $\cC'$. Therefore,
differentiating under the integral, we obtain (without expanding $G'$ for the moment)
\begin{equation}
\label{euler}
0=\frac{d}{d\varepsilon}I(u+\varepsilon\eta)|_{\varepsilon=0}=\int_a^{b} G'\bigl(
u(t)\bigr) \eta(t)\, dt.
\end{equation}
Focusing on the interval $[a,b]$, this condition can be interpreted as follows:
the restriction to $[a,b]$ of the function $G'(u(t))$ is orthogonal, in $L^2(a,b)$, to
every function $\eta\in L^\infty(a,b)$ satisfying \eqref{ortog*}.
Now, given $\eta\in L^2(a,b)$ satisfying \eqref{ortog*}, by a standard truncation
argument (and a slight perturbation of the resulting functions) one can easily
construct functions $\eta_k\in L^\infty(a,b)$, each satisfying
$\int_a^b t^{p-1} \eta_k(t)\,dt=0$ (hence also $\int_a^b G'(u(t))\eta_k(t)\,dt=0$),
such that $\eta_k\to\eta$ in $L^2(a,b)$. Letting $k\to\infty$,
we see that $G'(u(t))$ is orthogonal, in $L^2(a,b)$,
to any $\eta\in L^2(a,b)$ (not necessarily $L^\infty$) satisfying \eqref{ortog*}, and therefore
\begin{equation}\label{eq cinque*}
G'(u(t))=c\, t^{p-1}\quad\text{for a.e. $t\in [a,b]$}
\end{equation}
for some constant $c>0$. In fact, by the arbitrariness
of the interval  $[a,b]\subset (0,M)$, letting $a\to 0^+$ and $b\to M^-$
one can see
that \eqref{eq cinque*} holds
for a.e. $t\in (0,M)$,
and hence for \emph{every} $t\in(0,M)$ as $u(t)$ is decreasing. 
Defining a new constant $\lambda$ by $\lambda^{1-p}=c$, we can write this as
\begin{equation}\label{eqsei*}
G'(u(t))=(t/\lambda)^{p-1}\quad\forall t\in (0,M)
\end{equation}
and hence, since $G'(s)= \exp(-(s \,d!)^{1/d})$,
we find
\begin{equation}
\label{ulog2*}
u(t)=
\begin{cases}
\displaystyle \frac 1 {d!}\bigl(-\log (t/\lambda)^{p-1}\bigr)^d & \text{if  $t\in (0,M)$}\\
0\phantom{\displaystyle \frac I 2} & \text{if $t\in (M,A)$}
\end{cases}
\end{equation}
the second case being meaningful only if $M<A$.
Observe that, since $u(t)>0$ on $(0,M)$, we must have
\begin{equation}
\label{lgM*}
\lambda\geq M,
\end{equation} 
in particular $M<\infty$. Now we prove that, if $M<A$, then actually $\lambda=M$,
so that by \eqref{ulog2*} $u(t)$ is continuous also at $t=M$.
Arguing by contradiction, i.e. assuming that $M<A$ but $\lambda>M$,
%\begin{equation}
%\label{leftlimpos}
%\lim_{t\to M^-} u(t)=-\log(M/\lambda)^{p-1}>0,
%\end{equation}
we choose $\delta>0$
small enough (so that $M-\delta>0$ and $M+\delta<A$) and define the function
\begin{equation}
\label{defeta*}
\eta(t)=\begin{cases}
-t^{1-p} & \text{if $t\in (M-\delta,M)$}\\
t^{1-p} & \text{if $t\in (M,M+\delta)$}\\
0 &\text{otherwise}
\end{cases}
\end{equation}
in such a way that \eqref{ortog*} holds true.
%\begin{equation}
%\label{intzero}
%\int_0^A t^{p-1}\eta(t)\,dt=0.
%\end{equation}
Defining $u_\eps(t)=u(t)+\eps\eta(t)$, \eqref{defeta*} and \eqref{ulog2*} 
(combined with our assumption that $\lambda>M$)
imply that $u_\eps\geq 0$ on $(0,A)$, provided $\eps>0$ is small enough. This, combined with
\eqref{ortog*}, yields for small $\eps>0$ an admissible competitor $u_\eps\in \cC'$, so that
$I(u_\eps)\leq I(u)$ since $u$ is a maximizer of $I$. Hence, differentiating
under the integral as done for \eqref{euler}, we obtain
\[
0\geq \lim_{\varepsilon\to0^+}\frac{I(u_\eps)-I(u)}{\varepsilon}=\int_{M-\delta}^{M+\delta} G'(u(t))\eta(t)\, dt.
%=-\int_{M-\delta}^M G'(u(t)) t^{1-p}\,dt+
%\int_M^{M+\delta} G'(0) t^{p-1},
\]
According to \eqref{defeta*}, \eqref{eqsei*}, and the fact that $G'(u(t))=G'(0)=1$ when $t>M$,
the last inequality reduces to
\[
0\geq
-\int_{M-\delta}^M \lambda^{1-p}\,dt + \int_M^{M+\delta} t^{1-p}\,dt.
\]
Dividing by $\delta$ and letting $\delta\to 0^+$ we find 
 $M\geq \lambda$,  which is
the desired contradiction. Hence $\lambda=M$, if $M<A$.

This shows that, regardless of whether $M<A$ or $M=A$,  \eqref{ulog2*} can be written, with the notation \eqref{def log-}, as
\begin{equation}
\label{expru*}
u(t)=\frac 1 {d!} \bigl(\logmeno (t/\lambda)^{p-1}\bigr)^d \quad\forall t\in (0,A).
\end{equation}
Now let $h(\lambda)=p \int_0^A t^{p-1} u_\lambda(t)\,dt$, $\lambda>0$, where $u_\lambda$ is the function on the
right-hand side of \eqref{expru*}. Then $h(\lambda)$ is strictly increasing, so that there is at most one value of
$\lambda>0$ such that  \eqref{saturated*} (i.e. $h(\lambda)=B^p$) is satisfied, and this
proves the uniqueness of $u$, as a maximizer of $I$, via \eqref{expru*}.
To quantify $\lambda$ and make \eqref{expru*} more explicit, note that if $A<\infty$
(by the change of variable $t/A=\tau$)
\begin{equation}
\label{hA}
h(A)%\frac p{d!}\int_0^A t^{p-1} (-\log (t/A)^{p-1}\bigr)^d\,dt
=\frac {pA^p}{d!}\int_0^1 \tau ^{p-1} (-\log \tau^{p-1}\bigr)^d\,d\tau
=A^p \frac {(p-1)^d}{p^d}=\kp^d A^p.
\end{equation}
In case (ii), when $B^p \leq \kp^d A^p$, the strict monotonicity of $h(\lambda)$ and
the condition $h(\lambda)=B^p$ imply that $\lambda \leq A$.
%, and \eqref{expru*}
%amounts to
%\begin{equation}
%\label{ucaso<}
%u(t)=\begin{cases}
%\frac 1 {d!}\bigl(-\log (t/\lambda)^{p-1}\bigr)^d
%& \text{if $t\in (0,\lambda)$}\\
%0 & \text{if $t\in [\lambda,A).$}
%\end{cases}
%\end{equation}
%Extending $u(t)$.... 
Then, by \eqref{expru*}, \eqref{saturated*} becomes
\begin{align*}
B^p &= \frac p {d!} \int_0^\lambda t^{p-1} \bigl(-\log (t/\lambda)^{p-1}\bigr)^d\,dt\\
&=
\frac {p \lambda^p} {d!} \int_0^1 \tau^{p-1} \bigl(-\log \tau ^{p-1}\bigr)^d\,d\tau
=\lambda^p \frac{(p-1)^d}{p^d}=\lambda^p \kp^d,
\end{align*}
so that $\lambda=B\kp^{-d/p}$ and \eqref{ucasob} follows from \eqref{expru*}.
Since $G(0)=0$, using \eqref{expru*} we also have
\[
I(u)=%\int_0^A G(u(t))\,dt=
\int_0^\lambda G\bigl((-\log(t/\lambda)^{p-1})^d/d!\bigr)\,dt=
\lambda \int_0^1 G\bigl((-\log\tau^{p-1})^d/d!\bigr)\,d\tau
\]
and since by \eqref{defG}
\begin{equation}\label{Gu}
G\bigl((-\log\tau^{p-1})^d/d!\bigr)=1-\tau^{p-1} \sum_{j=0}^{d-1} \frac{(-\log \tau^{p-1})^j}{j!},
\end{equation}
one can check that the last integral evaluates to $\kp^d$, so that 
\[
I(u)=\lambda \kp^d=B \kp^{-d/p}\kp^d = B \kp^{d\kp},
\]
and \eqref{Icasob} follows.
 
Similarly, in case (iii), when 
$B^p > \kp^d A^p$, 
the condition $h(\lambda)=B^p$ implies that $\lambda > A$, and \eqref{expru*}
simplifies to \eqref{ucasoa}.
Then \eqref{saturated*} becomes
\[
B^p = \frac p {d!} \int_0^A t^{p-1} \bigl(-\log (t/\lambda)^{p-1}\bigr)^d\,dt\quad (\lambda> A),
\]
but now the values of $\lambda$ and $I(u)$ cannot be computed explicitly, except when $d=1$.
In this case, computing the last integral, one finds $\lambda=A e^{(B/A)^{p}/(p-1)-1/p}$
as claimed after \eqref{ucasoa}, and  since  when $d=1$  $G(s)=1-e^{-s}$ and $u(t)=-\log(t/\lambda)^{p-1}$
by \eqref{ucasoa}, we can compute 
\begin{equation*}
I(u)=
\int_0^A G(u(t))\,dt
=\int_0^A \left(1-(t/\lambda)^{p-1}\right)\,dt=A\left(1-\frac{(A/\lambda)^{p-1}}p\right),
\end{equation*}
and \eqref{Icasoad=1} follows if one replaces $\lambda$ with the value mentioned above.
\end{proof}

\begin{remark}\label{rem p=1} In Theorem \ref{thmcharmax}, when $p=1$ we
assumed $A<
\infty$.  Indeed, when $p=1$ and $A=
\infty$,
there is no extremal function: for every $u\in\cC\setminus\{0\}$ we have, by \eqref{eq due} (with $p=1,A=\infty$),
\[
I(u)=
\int_0^{+\infty} G(u(t))\, dt<\int_0^{+\infty} u(t)\, dt\leq B,
\]
whereas if we consider the sequence of functions $u_n$ on $(0,+\infty)$ given by $u_n(t)=B/n$ for $t\in (0,n)$, and $u_n(t)=0$ elsewhere, we have
$
I(u_n)=G(B/n)/n\to B$ as $n\to\infty$.
\end{remark}

\section{Proof of the main result in arbitrary dimension}\label{secdim}

This section is devoted to the proof of Theorems \ref{thm thm2 mult} and \ref{thm thm2}.
\begin{proof}[Proof of Theorems \ref{thm thm2 mult} and \ref{thm thm2}.] 
First observe that (i), (ii), and the last sentence of Theorem \ref{thm thm2} follow,
respectively,  from (i), (ii)
and the last sentence
of Theorem \ref{thm thm2 mult}, as particular cases when $d=1$. The same is true of (iii),
with the difference that \eqref{eq stima2} and \eqref{eq estr2} are more explicit than
\eqref{eq stima2 mult} and \eqref{eq estr2 mult 2}, since when $d=1$ the value of $\lambda$
can be found explicitly. Thus, we shall prove only Theorem \ref{thm thm2 mult}, discussing
how \eqref{eq stima2} and the value of $\lambda$ in \eqref{eq estr2} are obtained when $d=1$.

Let $F \in L^p(\rdd)$ 
satisfy \eqref{constraints},
and let $\mu$ be its distribution function as in \eqref{defmu}. Observing
that, due to \eqref{boundmu} and $G(0)=0$,
the integral in \eqref{stimanorma} can be restricted to $(0,A)$, recalling
\eqref{eq uno} and \eqref{defC} we have from \eqref{stimanorma}
\begin{equation}
\label{stimanorma5}
\Vert L_F\Vert_{L^2\to L^2}\leq \int_0^A G\bigl(\mu(t)\bigr)\,dt=I(\mu)\leq \sup_{v\in\cC} I(v)
\end{equation}
(note that $\mu\in\cC$).
On the other hand, this supremum is achieved by a unique function $u\in \cC$
completely characterized in Theorem \ref{thmcharmax}, so that
\begin{equation}
\label{stimanorma6}
\Vert L_F\Vert_{L^2\to L^2}\leq I(u) = \int_0^A G(u(t))\,dt.
\end{equation}
Then \eqref{eq stima2 ter 2} and \eqref{eq stima1 mult 2} follow, respectively, from \eqref{Icasoi}
and \eqref{Icasob}, while \eqref{eq stima2 mult} coincides with \eqref{stimanorma6} because, in
\eqref{eq stima2 mult}, $u_\lambda$ is the function $u$ in \eqref{ucasoa}. In particular, when $d=1$
and the value of $I(u)$ is the one in \eqref{Icasoad=1}, \eqref{stimanorma6} yields \eqref{eq stima2}.

Now observe that equality in \eqref{eq stima2 ter 2}, \eqref{eq stima1 mult 2} or 
\eqref{eq stima2 mult} (i.e. equality in \eqref{stimanorma6}) occurs if and only if
the two inequalities in \eqref{stimanorma5} reduce to equalitites, which corresponds
to the simultaneous validity of these two conditions:

\medskip

\noindent {(a)} Equality occurs in \eqref{stimanorma}. By Theorem \ref{thm4}, this means that
$F(z)=e^{i\theta} \rho(|z-z_0|)$ for some $\theta\in\bR$ and $z_0\in\rdd$, where $\rho:[0,+\infty)\to[0,+\infty)$ is decreasing.
Therefore, for every $t\in (0,A)$, the  set $\{|F|>t\}$ is a ball (centered at $z_0$) 
of measure $\mu(t)$ by \eqref{defmu}, whereas $\{|F|>t\}=\emptyset$ if $t\geq A$, by \eqref{constraints}.
%Thus, knowledge of $\mu$ allows one to determine $|F(z)|$ for a.e.

\smallskip

\noindent {(b)} $\mu$ (restricted to $(0,A)$ and regarded as an element of $\cC$) achieves the supremum in \eqref{eq uno}. By Theorem \ref{thmcharmax},
this means that $\mu$ is either the constant function equal to $B/A$, or the function in \eqref{ucasob}, or
the function in \eqref{ucasoa}, depending on whether we are in case (i), (ii) or (iii) (note that
these three cases match those of Theorem \ref{thm thm2 mult}), and moreover
\begin{equation}
\label{musat}
p\int_0^A t^{p-1}\mu(t)\,dt=B^p. 
\end{equation}

\medskip

Combining (a) and (b), we see that in case (i), where $\mu=B/A$ on $(0,A)$, for every
$t\in (0,A)$ the set $\{|F|>t\}$ is a ball 
(centered at $z_0$) of measure $B/A$, and hence  $F$ has the form
in \eqref{Fcasei}.

In case (ii), where $\mu$ is as in \eqref{ucasob} with $\lambda=B \kp^{-\frac d p}$, we can reconstruct the function
$\rho$ as follows.  For every $t\in (0,\lambda)$, the radius $r$ of a $2d$-ball
of measure $\mu(t)$ (such as the set $\{|F|>t\}$) satisfies the equation
\begin{equation}
\label{inversione}
\frac {\pi^d r^{2d}}{d!}  =\mu(t)=\frac{(-\log(t/\lambda)^{p-1})^d}{d!}
\end{equation}
(note that $\pi^d/d!$ is the volume of the unit ball in $\rdd$), and clearly $\rho(r)=t$.
Solving for $t$ we obtain
\begin{equation}
\label{solvet1}
t=\rho(r)=\lambda e^{-\frac \pi{p-1} r^2}\quad \forall t\in (0,\lambda),
\end{equation}
and this determines $\varphi(r)$ for every $r>0$. Since $F(z)=e^{i\theta}\varphi(|z-z_0|)$,
\eqref{eq estr1 mult} is proved.

Finally, in case (iii), where the restriction of $\mu(t)$ to $(0,A)$ is as in \eqref{ucasoa}, 
since now $\lambda > A$ we see that $\mu$ has a jump at $t=A$ (recall that  $\mu(t)=0$ when $t\geq A$). However,
if $t\in (0,A)$, to reconstruct $\varphi$ we can proceed as in \eqref{inversione}, 
now obtaining 
\begin{equation}
\label{solvet2}
t=\rho(r)=\lambda e^{-\frac \pi{p-1} r^2}\quad \forall t\in (0,A).
\end{equation}
The difference with respect to \eqref{solvet1} is that, since now $A<\lambda$, \eqref{solvet2} determines
$\rho(r)$ only for $r>r_0$, where $r_0$ is the value of $r$ that, plugged into \eqref{inversione},
yields $t=A$. However, since $\rho$ is decreasing and moreover
$\rho\leq A$ by \eqref{constraints}, we obtain that $\rho(r)=A$ when $r\in (0,r_0]$. In other words,
\[
\rho(r)=\min\left\{\lambda e^{-\frac \pi{p-1} r^2},A\right\},\quad r>0,
\]
and hence \eqref{eq estr2 mult 2} is proved. When $d=1$, the value of $\lambda$ is
the one mentioned after \eqref{ucasoa}, so that also \eqref{eq estr2} is established.

Notice that \eqref{saturF} follows from \eqref{defmu} and the formula
\[
\Vert F\Vert_{L^p}^p =p\int_0^\infty t^{p-1} \mu(t)\,dt=p\int_0^A t^{p-1} \mu(t)\,dt,
\]
combined with \eqref{musat}. Finally, the characterization of those $f,g$ satisfying \eqref{optfg} 
%(i.e. \eqref{fandg})
follows from Theorem \ref{thm4},  since any $F$ achieving equalitites  in \eqref{stimanorma5} achieves,
in particular, also equality in \eqref{stimanorma}.
\end{proof}

\section{Wavelet localization operators}\label{sec wavelet}
In this section we normalize the Fourier transform as
\[
\widehat{f}(\omega)= \frac{1}{\sqrt{2\pi}}\int_{\bR} e^{-i\omega t} f(t)\, dt.
\]
 For $\beta>0$ consider the so-called Cauchy wavelet $\psi_\beta\in L^2(\bR)$ \cite{daubechies_book,daubechies_paul,grossmann} defined by
\[
\widehat{\psi_\beta}(\omega)=\frac{1}{c_\beta}\chi_{[0,+\infty)}(\omega) \omega^\beta e^{-\omega},
\]
where $c_\beta>0$ and $c_\beta^2= 2\pi 2^{-2\beta} \Gamma(2\beta)$, so that $\|\widehat{\psi_\beta}\|^2_{L^2(\bR_+, d\omega/\omega)}=1/(2\pi)$  (${\bR}_+=(0,+\infty)$). This normalization is chosen for the corresponding wavelet transform, defined below, to be an isometry. Observe also that $\|\psi_\beta\|_{L^2}^2=\beta/(2\pi)$.

Let $H^2(\bR)$ be the Hardy space of functions in $L^2(\bR)$ whose Fourier transform is supported in $[0,+\infty)$, endowed with the $L^2$ norm (hence $\|f\|_{H^2}=\|f\|_{L^2}$). In particular, $\psi_\beta\in H^2(\bR)$. For $f\in H^2(\bR)$ we consider the corresponding \textit{wavelet transform} $\mathcal{W}_{\psi_\beta} f$, defined by
\[
\mathcal{W_{\psi_\beta}} f(x,y)= \frac{1}{\sqrt{y}}\int_{\bR} f(t) \overline{\psi_\beta\Big(\frac{t-x}{y}\Big)}\, dt
\]
for $(x,y)\in\bR\times\bR_+$.

Consider now the measure $d\nu=y^{-2}dxdy$, that is the left Haar measure on $\bR\times\bR_+\simeq\bC_+$ regarded as the ``$ax+b$" group.
Then, as anticipated, $\mathcal{W}_{\psi_\beta}: H^2(\bR)\to L^2(\bR\times\bR_+,d\nu)$ is an isometry (not onto); cf. \cite{grossmann}.

For $F\in L^p(\bR\times\bR_+,d\nu)$, $1\leq p\leq\infty$ and $f,g\in H^2(\bR)$ we define the \textit{wavelet localization operator} $L_{F,\beta}$ by
\[
\langle L_{F,\beta} f,g\rangle_{L^2}= \int_{\bR\times\bR_+}F \, \mathcal{W}_{\psi_\beta}f \,\overline{\mathcal{W}_{\psi_\beta} g}\, d\nu.
\]
Spectral properties of such operator have been studied by many authors; see e.g. \cite{daubechies_book,daubechies_paul,wong} and the references therein.
Our goal is to find optimal estimates
\begin{equation}\label{eq ast wav}
 \|L_{F,\beta}\|_{H^2\to H^2}\leq C=C(p,A,B),
\end{equation}
where $F\in L^p(\Hp,d\nu)$ is subject to the double constraint
\begin{equation}
\label{constraintsH}
\Vert F\Vert_{L^\infty(\Hp)}\leq A\qquad\text{and}\qquad
\Vert F\Vert_{L^p(\Hp,d\nu)}\leq  B
\end{equation}
(cfr. \eqref{stimaC} and \eqref{constraints}),
where $p$, $A$, and $B$ are as in Theorem \ref{thm thm2 mult}.

We will need the following result from \cite{joao_tilli}, which is the analogue of Theorem \ref{generalFK}
for the wavelet transform:
\begin{theorem}\label{thm wav0}
Let $\beta>0$. For every $\nu$-measurable set $\Delta\subset \bR\times\bR_+$ such that $\nu(\Delta)<\infty$, and every
$f\in H^2(\bR)$ such that $\|f\|_{L^2}=1$, we have
\[
\int_\Delta |\mathcal{W}_{\psi_\beta}f|^2\, d\nu \leq \GH\bigl({\nu(\Delta)}\big),
\]
where
\begin{equation}
\label{defGH}
\GH(s)=1-\Big(1+\frac{s}{4\pi}\Big)^{-2\beta},\quad s>0.
\end{equation}
Moreover, if $\nu(\Delta)>0$, equality occurs if and only if $f(t)$ has the form
\begin{equation}
\label{charfh}
t\mapsto \frac{c}{\sqrt{y}}\psi_\beta\Big(\frac{t-x_0}{y_0}\Big)
\end{equation}
for some $c\in\bC$ such that $|c|^2=2\pi/\beta$ and some $(x_0,y_0)\in\bR\times\bR_+$, and $\Delta$ is 
$\nu$-equivalent to a hyperbolic disc of center $z_0=x_0+ i y_0$.
\end{theorem}
We recall that in the Poincar\'e upper half-plane $\bC_+$, endowed with the hyperbolic metric $y^{-2}(dx^2+dy^2)$, the open hyperbolic disc of center $z_0$ and $\nu$-measure $s>0$ is given by
\begin{equation}\label{eq disc}
\Big|\frac{z-z_0}{z-\overline{z_0}} \Big|^2<1-\Big(1+\frac{s}{4\pi}\Big)^{-1}.
\end{equation}
This is most easily checked in the Poincar\'e unit disc $|w|<1$, with the metric $4(1-|w|^2)^{-2}|dw|^2$ and corresponding measure $4(1-|w|^2)^{-2} dA(w)$, which is isometric to the above upper half-plane via the map $w= (z-i)/(z+i)$ (by a M\"obius transformation one can reduce to the case $z_0=i$ and the hyperbolic disc $|z-i|/|z+i|<r$ is mapped to $|w|<r$).

To state the main result (the analogue of Theorem \ref{thm thm2} for the wavelet transform), it is
convenient to define the constants
\begin{equation}
\label{defsigma}
\sigma = \frac{p-1}{2\beta p+1},\quad \alpha=\frac{p-1}{2\beta+1}.
\end{equation}
\begin{theorem}\label{thm thm2 wav} Assume $\beta>0$, let $A$,  $B$ and $p$ be as in Theorem \ref{thm thm2 mult}, and consider $F\in L^p(\Hp,d\nu)$ satisfying \eqref{constraintsH}.
\begin{itemize}
\item[(i)] If $p=1$ (and hence $A<\infty$), then
\begin{equation}
\label{punoh}
\|L_{F,\beta}\|_{H^2\to H^2}\leq 
A \,\,\GH(B/A),
\end{equation}
and equality occurs if and only if
$
F=e^{i\theta} A\,\chi_{\cB}
$
for some $\theta\in\bR$, where $\cB\subset\bC_+$ is a 
hyperbolic disc of measure $\nu(\cB)=B/A$, centered at some $z_0\in \bC_+$.
\item[(ii)] If $p>1$ and $(B/A)^p\leq 4\pi\sigma $ (in particular if $A=\infty$), then 
\begin{equation}\label{eq stima1 wav}
\|L_{F,\beta}\|_{H^2\to H^2}\leq\frac{2\beta}{(4\pi)^{1/p}}\,\sigma^{\kp}B,
\end{equation}
with equality  if and only if
\begin{equation}\label{eq estr1 wav}
F(z)=e^{i\theta}\lambda\Big(1-\Big|\frac{z-z_0}{z-\overline{z_0}} \Big|^2  \Big)^{-\alpha },\qquad z\in\bC_+,
\end{equation}
for some $\theta\in\bR$ and $z_0\in\bC_+$, with $\lambda$ as in \eqref{eq lambda 1}.
\medskip
\item[(iii)] If $p>1$ and $(B/A)^p> 4\pi\sigma$, then
\begin{equation}\label{eq stima2 wav}
\|L_{F,\beta}\|_{H^2\to H^2}\leq A\Big(1-p^{2\beta} \Big(\frac\sigma\alpha\Big) ^{2\beta+1}\Big(1+\frac{(B/A)^p}{4\pi}\Big)^{-2\beta} \Big),
\end{equation}
with equality
 if and only if
\begin{equation}\label{eq estr2 wav}
F(z)=e^{i\theta}\min\Big\{\lambda\Big(1-\Big|\frac{z-z_0}{z-\overline{z_0}} \Big|^2 \Big)^{-\alpha},A\Bigr\},\quad z\in\bC_+,
\end{equation}
for some $\theta\in\bR$ and $z_0\in\bC_+$, with $\lambda$ as in \eqref{eq lambda 2}.

\medskip

\end{itemize}
In the above cases where equality occurs,  we have $\Vert F\Vert_{L^p}=B$,
and there holds $\|L_{F,\beta}\|_{H^2\to H^2}=|\langle L_{F,\beta} f,g\rangle|$ 
for some $f,g\in H^2(\bR)$ such that $\|f\|_{L^2}=\|g\|_{L^2}=1$, if and only if 
both $f$ and $g$ have the form as in \eqref{charfh}, possibly with different $c$'s, but with the
same $(x_0,y_0)\in\bR\times\bR_+$ defined by $z_0=x_0+iy_0$.
% we have
%\begin{equation}\label{eq deff}
%f(t)=\frac{c}{\sqrt{y}}\psi_\beta\Big(\frac{t-x}{y}\Big),\qquad g(t)=c' f(t).
%\end{equation}
\end{theorem}
In particular, for $p=1$ we therefore obtain the estimate
\[
\|L_{F,\beta}\|_{H^2\to H^2}\leq \|F\|_{L^\infty} \Big(1-\Big(1+\frac{\|F\|_{L^\infty}/\|F\|_{L^1}}{4\pi}\Big)^{-2\beta}\Big),
\]
which is the wavelet counterpart of \eqref{eq galbis}.

As for the short-time Fourier transform, an intermediate step towards the proof of this theorem is the
following  result (the analogue of Theorem \ref{thm4} for the wavelet transform):
\begin{theorem}
\label{thm4h}
Assume $F\in L^p(\Hp,d\nu)$ for some $p\in [1,\infty)$, and let 
\begin{equation}
\label{defmuh}
\mu(t)=\nu\left(\left\{ F(z)|>t\right\}\right),\quad t>0,
\end{equation}
denote the distribution function of $|F|$. Then
\begin{equation}
\label{stimanormaH}
\Vert L_{F,\beta}\Vert_{H^2\to H^2}\leq \int_0^\infty \GH\bigl(\mu(t)\bigr)\,dt
\end{equation}
where $\GH$ is as in \eqref{defGH}, with
equality if and only if $F(z)=e^{i\theta}|F(z)|$
 and
$|F(z)|=\rho(|z-z_0|^2/|z-\overline{z_0}|^2)$, for some $\theta\in\bR$, some $z_0\in\Hp$ and
some decreasing function
$\rho:[0,1)\to [0,+\infty)$.
% someis (up to multiplication by a unimodular
%constant) nonnegative, radially symmetric around some $z_0\in \bR^{2d}$, and radially decreasing.
In this case,  there holds
\begin{equation}
\label{fandgh}
\left\vert\langle L_F f,g\rangle\right\vert =\Vert L_F\Vert_{L^2\to L^2}
\end{equation}
for some $f,g\in H^2(\bR^d)$ such that $\Vert f\Vert_{L^2}=\Vert g\Vert_{L^2}=1$, if (and only
if, when  $F$ is not identically zero) both $f$ and $g$ are functions of the kind \eqref{charfh}, possibly with different 
$c$'s, but with the same $(x_0,\omega_0)$ defined by $x_0+iy_0=z_0$.
\end{theorem}

The proof of the last two theorems follows the same pattern as the proof of Theorems \ref{thm thm2}
and \ref{thm4}, through the study of a variational problem similar to \eqref{eq uno}
and results analogous to those of Section \ref{sec due}. We omit the details, limiting ourselves to
describe how the proofs can be adapted to the hyperbolic setting.

Apart from obvious changes (such as replacing $\rdd$ with $\Hp$ and $dz$ with $d\nu(z)$ in the
integrals etc.), Theorem \ref{thm4h} is proved by the same slicing argument as Theorem \ref{thm4}, 
now using 
Theorem \ref{thm wav0}
in place of Theorem \ref{generalFK}, so that $F$ achieves
equality in \eqref{stimanormaH} if and only if $F(z)=e^{i\theta}|F(z)|$
and all the superlevel sets $\{|F|>t\}$ are hyperbolic discs  centered at some $z_0\in \Hp$.
Since every such disc has the form as in \eqref{eq disc}, reconstructing (for a.e. $z$)
$|F(z)|$ 
as in
\eqref{lcake}, we see that $|F(z)|$ only depends on $|z-z_0|^2/|z-\overline{z_0}|^2$, and hence
$|F(z)|=\rho(|z-z_0|^2/|z-\overline{z_0}|^2)$ for some (necessarily decreasing) function
$\rho:[0,1)\to [0,+\infty)$, as claimed.

Subsequently, to estimate the right hand side of \eqref{stimanormaH}, one studies the variational
problem
\begin{equation}\label{eq unoh}%anche \label{eq tre}
\sup_{v\in\cC}I(v)\quad\text{where}\quad
    I(v):= \int_0^{A} \GH(v(t))\, dt
\end{equation}
(the analogue of \eqref{eq uno}, with the same $\cC$ as in \eqref{defC}), and
proceeding as in Section \ref{sec due} one proves a result similar
to Theorem \ref{thmcharmax}, i.e. the existence of a unique maximizer $u\in\cC$, 
that we shall identify distinguishing
three different cases as
(i), (ii) and (iii) of Theorem \ref{thm thm2 wav}.
Then, by \eqref{stimanormaH},
\begin{equation}
\label{stimanormas}
\|L_{F,\beta}\|_{H^2\to H^2}\leq I(\mu)\leq \sup_{v\in\cC} I(v)=I(u).
\end{equation}
When $p=1$ (case (i))
 $u=B/A$ is constant, again
by Jensen inequality (note that $\GH(s)$ is strictly convex), and this eventually leads to \eqref{punoh}.

When $p>1$, arguing exactly as for \eqref{eq cinque*} one proves
that $u$ satisfies $\GH'(u(t))=c t^{p-1}$ on the interval $(0,M)$ where
 $u>0$, which 
can be rewritten as $\GH'(u(t))=2\beta( t/\lambda)^{p-1}$ on $(0,M)$
for some  $\lambda>0$
 (cfr. \eqref{eqsei*}). Then, reasoning as after \eqref{lgM*}, one proves the continuity of $u$ on $(0,A)$ also when $M<A$, so that
 computing $\GH'(s)$
from \eqref{defGH} and solving for $u(t)$, recalling \eqref{defsigma} one eventually obtains
 \begin{equation}
 \label{eq u wav}
 u(t)=u_\lambda(t)= 4\pi\,\max\{(t/\lambda)^{-\alpha}-1,0\},\quad t\in (0,A),
 \end{equation}
cfr. \eqref{expru*}. The value of
 $\lambda>0$ can be uniquely determined by the condition that the constraint
 \eqref{eq due} is saturated, i.e. by solving the equation
 \begin{equation}\label{eq const wav}
B^p=h(\lambda):=  p\int_0^A t^{p-1} u_\lambda(t)\, dt
 \end{equation}
(again, $h$ is strictly increasing).
Since now (cfr. \eqref{hA}), with the notation as in \eqref{defsigma}, for $A<\infty$ we have
\[
h(A)=4\pi p\int_0^A t^{p-1} \Bigl((t/A)^{-\alpha}-1\Bigr)\,dt=
\frac{4\pi(p-1)}{2\beta p+1} A^p=4\pi\sigma A^p,
\]
we see from \eqref{eq const wav} that $\lambda\leq A$ when 
$(B/A)^p\leq 4\pi\sigma$ (case (ii)), while $\lambda>A$
when $(B/A)^p> 4\pi\sigma$ (case (iii)). In case (ii), by \eqref{eq u wav},
\eqref{eq const wav} becomes
\[
B^p=h(\lambda)=4\pi p\int_0^\lambda t^{p-1} \Bigl((t/\lambda)^{-\alpha}-1\Bigr)\,dt=
4\pi\sigma \lambda^p,
\]
whence 
\begin{equation}
\label{eq lambda 1}
\lambda=B (4\pi\sigma)^{-1/p}.
\end{equation}
Moreover, since $\GH(0)=0$, an explicit computation based on \eqref{eq unoh} and \eqref{eq u wav} gives
 \begin{align*}
 I(u_\lambda(t))&= \int_0^{\lambda} \Big(1-\big((t/\lambda)^{-\alpha}\big)^{-2\beta}\Big)\, dt
 =\lambda\int_0^1 \Big(1-\tau^{2\alpha\beta} \Big)\, d\tau\\
  &=\lambda\left(1-\frac 1 {2\alpha\beta+1}\right)=\frac{2\beta}{(4\pi)^{1/p}}\,\sigma^{1-1/p}B,
 \end{align*}
 and \eqref{eq stima1 wav} follows from \eqref{stimanormas}.

Similarly, in case (iii) where $(B/A)^p>  \frac{4\pi(p-1)}{2\beta p+1}$ and $\lambda>A$, 
\eqref{eq u wav} simplifies to $ u(t)=4\pi\bigl((t/\lambda)^{-\alpha}-1\bigr)$,
%\quad t\in (0,A)
so that \eqref{eq const wav} becomes
\[
B^p=4\pi p \int_0^A t^{p-1} \bigl((t/\lambda)^{ -\alpha}-1\bigr)\,dt
=4\pi A^p \left(\frac {p\sigma}{\alpha} \left(\frac A\lambda\right)^{-\alpha}-1
\right)
\]
whence
\begin{equation}
\label{eq lambda 2}
\lambda=A\left(\frac{4\pi A^p p\sigma}{\alpha(B^p+4A^p\pi)}\right)^{-\frac 1\alpha}
\end{equation}
and, by an elementary computation,
\[
I(u_\lambda)=\int_0^{A} \Big(1-\big((t/\lambda)^{-\alpha}\big)^{-2\beta}\Big)\, dt= A\Big(1-p^{2\beta} \Big(\frac\sigma\alpha\Big)^{2\beta+1}\Big(1+\frac{(B/A)^p}{4\pi}\Big)^{-2\beta} \Big).
\]
Then \eqref{eq stima2 wav} follows from \eqref{stimanormas}.

Finally, $F$ attains equality in \eqref{punoh}, \eqref{eq stima1 wav}, or \eqref{eq stima2 wav}
(i.e. equalities in \eqref{stimanormas}), if and only if equality occurs in \eqref{stimanormaH} and
$\mu(t)$ coincides with $u(t)$ (extended to $u(t)=0$ for $t\geq A$, if $A<\infty$). This, by Theorem \ref{thm4h},
means that $F=e^{i\theta} |F|$ and $|F(z)|=\rho(|z-z_0|^2/|z-\overline{z_0}|^2)$  (where $z_0\in\Hp$
and $\rho:[0,1)\to[0,+\infty)$ is decreasing), thus we can reconstruct $|F(z)|$ by knowledge of its distribution
function $\mu(t)=u(t)$. 
Indeed, 
since for every $t\in (0,A)$ the superlevel sets $\{|F|>t\}$
(being hyperbolic discs centered at $z_0$) have the form as in
\eqref{eq disc} with $s=u(t)$, i.e.
\[
\{|F(z)|>t\}=\left\{z\in \Hp\quad\text{such that}\quad\Big|\frac{z-z_0}{z-\overline{z_0}} \Big|^2< 1-\Big(1+\frac{u(t)}{4\pi}\Big)^{-1}\right\},
\]
using the layer cake representation \eqref{lcake} (with $m=A$ and $z\in\Hp$) one obtains,
for a.e. $z\in\Hp$,
\[
|F(z)|=\sup\left\{t\in (0,A) \quad\text{such that}\quad\Big|\frac{z-z_0}{z-\overline{z_0}} \Big|^2< 1-\Big(1+\frac{u(t)}{4\pi}\Big)^{-1} 
\right\}
\]
In case (i), when $p=1$ and $u(t)=B/A$ is constant on $(0,A)$, 
 we see that $F(z)$
is as claimed after \eqref{punoh}.
On the other hand, when $p>1$, an elementary computation based on \eqref{eq u wav} reveals that
\[
|F(z)|=\min \left\{
\lambda\Big(1-\Big|\frac{z-z_0}{z-\overline{z_0}} \Big|^2 \Big)^{\frac 1
\alpha},A\right\}.
\]
In case (ii), where $\lambda\leq A$ (as given in \eqref{eq lambda 1}), since $F=e^{i\theta}|F|$ one
obtains \eqref{eq estr1 wav}. Similarly, in case (iii) where $\lambda > A$ (as given in \eqref{eq lambda 2}),
one ontains \eqref{eq estr2 wav}.

\bibliographystyle{abbrv}
\bibliography{biblio.bib}

\end{document}